\documentclass[11pt]{amsart}

\usepackage{amsmath}
\usepackage{amssymb}
\usepackage{amsthm}
\usepackage{mathabx}

\usepackage{pb-diagram}
\usepackage{enumerate}
\newtheorem{thm}{Theorem}[section]
\newtheorem{cor}[thm]{Corollary}
\newtheorem{lem}[thm]{Lemma}
\newtheorem{prop}[thm]{Proposition}

\newtheorem*{con-no}{Principle}
\newtheorem*{prop-no}{Proposition}
\newtheorem{tham}{Theorem}

\theoremstyle{remark}
\newtheorem*{rem}{Remark}

\theoremstyle{definition}

\begin{document}

\pagestyle{myheadings}
\address{Department of Mathematics, Technion - Israel Institute of Technology, 32000, Haifa, Israel.}
\email{max@tx.technion.ac.il}
\date{\today}
\title[Ladder representations and Jacquet's conjecture]{On a local conjecture of Jacquet, ladder representations and standard modules}
\author{Maxim Gurevich}

\begin{abstract}
Let $E/F$ be a quadratic extension of p-adic fields. We prove that every smooth irreducible ladder representation of the group $GL_n(E)$ which is contragredient to its own Galois conjugate, possesses the expected distinction properties relative to the subgroup $GL_n(F)$. This affirms a conjecture attributed to Jacquet for a large class of representations. Along the way, we prove a reformulation of the conjecture which concerns standard modules in place of irreducible representations.
\end{abstract}

\maketitle

\section{Introduction}
Let $E/F$ be a quadratic extension of $p$-adic fields. Denote the group $G_n= GL_n(E)$. Let $g\mapsto g^\tau$ be the Galois involution on $G_n$ relative to the extension $E/F$. Denote $H_n = G_n^\tau = GL_n(F)$. We study the $H_n$-invariant functionals on admissible representations of $G_n$. In particular, we are interested to know whether a non-zero such functional exists for a given representation. In this case we will say that the representation is \textit{distinguished}. 
\par Let $\eta$ denote the quadratic character of $F^\times$ related to the extension $E/F$. We will call an admissible $G_n$-representation $\eta$-\textit{distinguished}, if a non-zero $H_n,\eta(\det)$-equivariant functional exists on it. 
\par It is long known (\cite{Fli91}) that a distinguished irreducible smooth representation $\pi$ of $G_n$ must satisfy $\pi^\vee\cong \pi^\tau$, where the left-hand side of the equation is the contragredient representation, while the right-hand side is the twist induced on representations by the involution $\tau$. The prediction, that certain variations of the converse implication should hold, are often referred to as Jacquet's conjecture. Let us formulate it as a general principle.
\par\textit{ An irreducible smooth representation of $G_n$ which satisfies\footnote{Sometimes an additional assumption is added which requires the central character of $\pi$ to be trivial on $F^\times$. Note, that the counter-example below remains valid with this assumption.} $\pi^\vee\cong \pi^\tau$, should be either distinguished or $\eta$-distinguished.}

\par In general this formulation is evidently false. As a counter-example, one can choose an $\eta$-distinguished irreducible supercuspidal representation $\rho$ of $G_2$, and look on the representation $1\times \rho$ of $G_3$, where $1$ is the trivial representation of $E^\times$ and the multiplication is in the sense of parabolic induction.
\par Yet, for significant large families of irreducible representations the principle above was indeed shown to hold. All discrete series are such, as proved by Kable in \cite{Kab}. Matringe's results in \cite{matr-unit} implied the same for certain other unitarizable representations, including the so-called Speh representations. 
\par In this work, we extend the validity of Jacquet's conjecture to the class of ladder representations, which was introduced by Lapid and M\'{i}nguez in \cite{LM}. This wide family of irreducible representations of $G_n$ includes discrete series and Speh representations as special cases.
\begin{tham}[Theorems \ref{ladder-thm-2} and \ref{noboth-thm}]\label{intro-1}
A ladder representation of $G_n$ which satisfies $\pi^\vee\cong \pi^\tau$ is either distinguished or $\eta$-distinguished. Moreover, a proper ladder representation of $G_n$ cannot be both distinguished and $\eta$-distinguished.
\end{tham}
The second statement of the theorem is again an expected property which was previously established for discrete series and Speh representations.

The precise statements of Theorems \ref{ladder-thm-2} and \ref{noboth-thm} go further by fully characterizing distinction and $\eta$-distinction of ladder representations from the combinatorial properties of the defining data of the representation.

\par For the proof of Theorem \ref{intro-1}, we turn to the class of reducible admissible representations of $G_n$ called standard modules. These are representations constructed by parabolically inducing a tempered representation $\kappa$ of a Levi subgroup $M< G_n$ twisted by an unramified character $\alpha$ of $M$ chosen from a certain cone. A standard module is uniquely defined by the triple $(M,\kappa,\alpha)$. Recall that the Langlands classification describes each irreducible smooth representation $\pi$ of $G_n$ as a unique irreducible quotient of a standard module $\Sigma(\pi)$. 
\par Thus, studying invariant functionals on an irreducible representation can be done by constructing such functionals on the corresponding standard module, and then determining whether they factor through the irreducible quotient. Such methods were explored in \cite{FLO} for studying distinction relative to unitary subgroups (in place of our $H_n$).
\par Note, that an irreducible smooth representation $\pi$ is generic, if and only if, the standard module $\Sigma(\pi)$ is irreducible (equivalently, $\pi\cong\Sigma(\pi)$).  We propose and prove the following reformulation of Jacquet's conjecture for all smooth irreducible representations, which coincides with the original formulation on generic representations. 
 
\begin{tham}[Corollary \ref{cor-firdir} and Lemma \ref{sec-dir}]\label{intro-2}
Suppose that $\pi$ is a smooth irreducible representation of $G_n$, whose standard module $\Sigma(\pi)$ is distinguished. Then, $\pi^\vee\cong \pi^\tau$ holds.
\par Conversely, let $\pi$ be a smooth irreducible distinguished representation of $G_n$ of pure type, i.e.~the supercuspidal support of $\pi$ is contained in the set $\{\rho \otimes|\det |_E^n\,:\,n\in\mathbb{Z}\}$ for some supercuspidal representation $\rho$. If $\pi^\vee\cong \pi^\tau$ holds, then the standard module $\Sigma(\pi)$ is either distinguished or $\eta$-distinguished.
\end{tham}
This result is then further extended in Theorem \ref{jac-mod} to a description of the same situation without the assumption on pure type. It again can be seen to coincide with a known statement (\cite[Theorem 5.2]{matr-gen}) when dealing with generic representations.
\par The feature which makes ladder representations approachable to our discussion is Theorem 1 of \cite{LM}. It states that when an irreducible $\pi$ is a ladder representation, the kernel of the quotient map $\Sigma(\pi)\to \pi$ can itself be described in terms of standard modules. Thus, in order to claim that an invariant functional on $\Sigma(\pi)$ factors through the quotient, it is enough to know it must vanish on the standard modules which generate the above kernel. This is the method by which we manage to deduce Theorem \ref{intro-1} from Theorem \ref{intro-2}.
\par The paper is organized as follows. Section 2 brings into our setting some known tools for studying distinction problems. The foremost tool is the geometric lemma of Bernstein-Zelevinsky which is long-known to serve as the Mackey theory of admissible representations. It allows us to study invariant functionals on induced representations through distinction properties of the inducing data. On top of that, our main tool for producing functionals on induced representations is the Blanc-Delorme theory developed in \cite{BD}. Their method, adapted to our needs in Proposition \ref{B-D}, can construct a desired functional by taking a continuation of an analytic family of integrals. 
\par Section 3 deals with the proof of both implications of Theorem \ref{intro-2}. We present each standard module as a multiplication, in the sense of parabolic induction, of essentially square-integrable representations. We then can state arguments of a combinatorial nature, such as Lemma \ref{key-lem}, about the structure of the space of invariant functionals on a standard module. The first implication of Theorem \ref{intro-2}, which is also the crucial step for the proof of Theorem \ref{intro-1}, follows from this key lemma. We also deduce a multiplicity one theorem (Proposition \ref{mult-one-prop}) for $H_n$-invariant functionals on a large class of (possibly reducible) standard modules.
\par The second implication is shown by obtaining the existence of invariant functionals from the Blanc-Delorme theory.  
\par Finally, Section 4 deals with ladder representations. We show the deduction of Theorem \ref{intro-1} from Theorem \ref{intro-2} as described above. Theorem \ref{ladder-thm} further resolves between distinction and $\eta$-distinction of a proper ladder representation, in terms of the inducing data of its standard module. 
\par For the last part of Theorem \ref{intro-1} (Theorem \ref{noboth-thm}), the additional tool of Gelfand-Kazhdan derivatives needs to be introduced. The idea, which traces back to \cite{Kab} and \cite{Fli93}, is that distinction of a given smooth irreducible representation implies the distinction of at least one of its derivatives. Using this in combination with the derivative data of ladder representations, which was obtain in \cite{LM}, allows us to contradict the possibility of distinction of certain representations.

\subsection*{Acknowledgements}
The work reported here is part of my Ph.D.~research. I would like to thank my advisor, Omer Offen, for suggesting the problem, sharing his insights and providing much guidance. Thanks are also due to Erez Lapid, Nadir Matringe and Yiannis Sakellaridis who all supplied useful remarks during the work on this project.


\section{Distinction of induced representations}

\par We will write representations of a locally compact totally disconnected group $G$ as $(\pi,V)$, or simply as $\pi$, where $V$ is a complex vector space and $\pi:G\to GL(V)$ is a homomorphism. The representation $(\pi,V)$ is called (smooth) admissible if the stabilizer of each vector in $V$ is an open subgroup of $G$ and for every compact open subgroup $K<G$, the space of vectors in $V$ invariant under $K$ is of finite dimension. 
\par Given an admissible representation $(\pi,V)$ of $G_n$ and a character $\alpha$ of $H_n$, we say that a functional $\ell$ on $V$ is $(H_n,\alpha)$-equivariant if $\ell(\pi(h)v) = \alpha(h)\ell(v)$ for all $h\in H_n$, $v\in V$. If such $\pi$ (not necessarily irreducible) has a non-zero $(H_n,\alpha)$-equivariant functional on it, we will say that $\pi$ is \textit{$\alpha$-distinguished}. Note, that $n$ will sometimes be implicit in our notation. We will say that $\pi$ is \textit{distinguished}, if it is $1$-distinguished.
\par For an admissible representation $\pi$ of $G_n$, we denote by $\pi^\vee$ its contragredient representation. Also, denote by $\pi^\tau$ its Galois twist, that is, $\pi^\tau(g)=\pi(g^\tau)$, for all $g\in G_n$. 

Given a character $\beta$ of $G_n$, we will often write $\beta\pi$ for the tensor product representation $\beta\otimes \pi$ of $G_n$.

\subsection{Consequences of the Geometric Lemma}
Any standard Levi subgroup of $G_n$ is of the form $M=M_{(m_1,\ldots, m_t)}=G_{m_1}\times\ldots \times G_{m_t} < G_n$ ($\sum_{i=1}^t m_i= n$), with the obvious diagonal embedding. If $\sigma_1,\ldots,\sigma_t$ are admissible representations of $G_{m_1},\ldots, G_{m_t}$, respectively, we denote by $\sigma_1\times\cdots\times  \sigma_t$ the representation of $G_n$ constructed by normalized parabolic induction. In other words, the $M$-representation $\sigma:=\sigma_1\otimes\cdots\otimes \sigma_t$ is naturally lifted to $P$, where $M\subseteq P\subseteq G_n$ is the standard parabolic subgroup corresponding to $M$. Then,
$$\sigma_1\times\cdots \times \sigma_t =  ind_P^{G_n} (\delta_P^{1/2}\sigma),$$
where $\delta_P$ is the modular character of $P$ and $ind$ denotes the (non-normalized) induction functor of smooth representations from the subgroup $P<G_n$.
For an admissible representation $\pi$ of $G_n$, we denote by $r_{M,G_n}(\pi)$ the representation of $M$ which is the normalized Jacquet module of $\pi$. See, for example, \cite[Section 2.3]{BZ1} for the definitions of induction of smooth representations and the Jacquet module. 
\par First, we would like to deal with distinction properties of representations of $G_n$ that are parabolically induced from a standard Levi subgroup. Let us fix one such subgroup $M<G_n$ and its corresponding parabolic subgroup $P$ for the rest of this section. We will need a convenient description of the double cosets space $P\setminus G_n/H_n$.
\par Let $W$ denote the Weyl group of $G_n$ realized as $N_{G_n}(T)/T$, where $T$ is the diagonal maximal torus of $G_n$. Set $W^M =\{w\in W\,:\, wMw^{-1}=M\}$, and its normal subgroup $W_M = N_M(T)/T$, which is the Weyl group of $M$. There is a natural mapping $p_M: W^M\to S_t$, whose kernel is $W_M$. Here $S_t$ is the permutation group on $\{1,\ldots,t\}$. It will sometimes be convenient to have the notation $\mathfrak{J}_M = \{1,\ldots,t\}$ and refer to the image of $p_M$ as permutations on $\mathfrak{J}_M$. 
\par Let $W[M]\subset W$ be the set of representatives of minimal length for the double cosets space $W_M\setminus W/W_M$. Let $W_2[M]\subset W[M]$ denote the subset of involutions inside it. By \cite[Proposition 20]{JLR} there is a bijection between $W_2[M]$ and $P\setminus G_n/H_n$. Explicitly, each $P-H_n$ double coset has a representative $\eta$ for which $\xi=\eta^\tau\eta^{-1}$ belongs to the normalizer $N_{G_n}(T)$. The representative can be chosen so that the projection of $\xi$ to $W$ will fall inside $W_2[M]$ (see also \cite[Lemma 19]{JLR}). The resulting involutive permutation is uniquely defined by the double coset. See also \cite[Section 3]{matr-gen} for an equivalent description of the double cosets in different terms.
\par It follows easily from Hilbert's Theorem 90 that for some $d\in T$, $d^\tau\xi d^{-1}= (d\eta)^\tau(d\eta)^{-1}$ is in fact a permutation matrix (consists only of $1$ and $0$ entries). Thus, for each $w\in W_2[M]$, we fix a representative $\eta_w$ of the associated $P-H_n$ double coset, for which $\eta_w^\tau \eta_w^{-1}$ is a permutation matrix (given by $w$).
\par An element $w\in W_2[M]$ and the double coset associated with it will be called \textit{$M$-admissible} if $w\in W^M$. Applying $p_M$, we see that the $M$-admissible double cosets are in natural correspondence with involutive permutations $\epsilon$ on $\mathfrak{J}_{M}$, for which $m_{\epsilon(i)}=m_i$ holds for all $i\in \mathfrak{J}_M$. The rest of the double cosets can still be described in similar terms, but by descending to a smaller Levi subgroup. Namely, for $w\in W[M]$, we need to observe the subgroup $M(w):= M\cap wMw^{-1}$ which must be a standard Levi subgroup for $G_n$. Note, that each $w\in W_2[M]\subseteq W_2[M(w)]$ is $M(w)$-admissible.
\par Looking at any inclusion $L\subseteq M=M_{(m_1,\ldots,m_t)}$ of standard Levi subgroups, we can describe $L$ as $M_\gamma$ with $\gamma = (l_{11},\ldots,l_{1s_1},\ldots,l_{t1},\ldots,l_{ts_t})$ and $\sum_{j=1}^{s_i} l_{ij} = m_i$, for all $i$. With this in mind, we give another natural enumeration of the blocks of $L$ as pairs
$$\mathfrak{J}_{L,M} = \{(i,j)\::\: i=1,\ldots, t,\,j=1,\ldots,s_i\}$$
with the lexicographical ordering. Naturally, the ordered set $\mathfrak{J}_{L,M}$ is identified with $\mathfrak{J}_L$ by sending $(i,j)$ to $\sum_{k=1}^{i-1}s_k+j$.
\par Given $w\in W_2[M]$, we have a description of $\epsilon_w:=p_{M(w)}(w)$ as an involution on $\mathfrak{J}_{M(w)}\cong \mathfrak{J}_{M(w), M}$. It must satisfy the rule that for each $1\leq i\leq t$ and $1\leq j<k\leq s_i$,  if $\epsilon_w(i,j)=(i'_1,j'),\,\epsilon_w(i,k)=(i'_2,k')$, then $i'_1<i'_2$. In fact, going over all standard Levi subgroups $L\subseteq M$ and all involutions $\epsilon$ of $\mathfrak{J}_{L,M}$, which satisfy the above condition and for which $l_{\epsilon(i,j)}=l_{i,j}$ holds for all $(i,j)\in \mathfrak{J}_{L,M}$, would give a full description of $W_2[M]$.
\par For each $w\in W_2[M]$, we define the subgroup $M^w:= M\cap \eta_w H_n\eta_w^{-1}$, and similarly for $P^w$. Note, that,
$$M^w = M(w)^w = \left\{ (g_i)_{i\in\mathfrak{J}_{M(w)}}\in M(w)\::\: g_{\epsilon_w(i)} = g^\tau_i\right\}.$$
\par Suppose that $\sigma$ is an admissible representation of $M$. Let $(\pi,V)$ be the representation parabolically induced from $\sigma$ to $G_n$. For $w\in W_2[M]$, let us define the $H_n$-representation $\mathcal{V}_w(\pi):= ind^{H_n}_{\eta_w^{-1} P\eta_w\cap H_n} \left(\delta_P^{1/2}\sigma|_{P^w}\right)^{\eta_w}$, where $(\cdot)^{\eta_w}$ denotes the conjugation functor that transfers a 
$P\cap \eta_w H_n\eta_w^{-1}$-representation into a $\eta_w^{-1} P\eta_w \cap H_n$-representation. Mackey theory (proved in the Geometric Lemma of \cite{BZ1}) gives a filtration of $\pi$ by $H_n$-sub-representations $\{0\}=V_0\subset V_1\subset\ldots\subset V_k=V$, in such a way that each subquotient $V_i/V_{i-1},\,i=1,\ldots,k$ is isomorphic to $\mathcal{V}_{w_i}(\pi)$ for some enumeration $(w_i)$ of $W_2[M]$.
\par The above geometric decomposition allows us to study the invariant functionals on $\pi$ in terms of distinction properties of certain Jacquet modules of $\sigma$. 
\begin{lem}\label{geom-functional}
Let $w\in W_2[M]$ be an involution, and suppose that the Jacquet module $r_{M(w),M}(\sigma)$ is a pure tensor representation, that is, 
$$r_{M(w),M}(\sigma) = \otimes_{i\in \mathfrak{J}_{M(w)}} \sigma_i.$$ 
\par Let $\mathfrak{F}\subset \mathfrak{J}_{M(w)}$ be a choice of representatives for the orbits of $\epsilon_w$ on $\mathfrak{J}_{M(w)}$ (i.e. one index out of $\{i,\epsilon_w(i)\}$ belongs to $\mathfrak{F}$, for all $i\in\mathfrak{J}_{M(w)}$). Then,
$$Hom_{H_n}(\mathcal{V}_w(\pi),\mathbb{C}) \cong$$ $$\cong \left(\otimes_{i\in\mathfrak{F}\,:\,\epsilon_w(i)=i} Hom_{H_{m_i}}(\sigma_i,\mathbb{C})\right)\otimes\left(\otimes_{i\in\mathfrak{F}\,:\,\epsilon_w(i)\neq i} Hom_{G_{m_i}}(\sigma_i^\tau,\sigma^\vee_{\epsilon_w(i)})\right).$$
\end{lem}

\begin{proof}
By \cite[Lemma 6.4]{FLO}\footnote{ The proof in \cite{FLO} adapts the results of \cite{LR} to its own setup. In particular, their $W_2[M]$ is defined differently. Yet, going through the same proof verbatim with our definitions would give the same result for our setting.} we deduce that 
$$Hom_{H_n}(\mathcal{V}_w(\pi),\mathbb{C})\cong Hom_{M^w}(r_{M(w),M}(\sigma),\mathbb{C}).$$ 
Now, from the description of $M^w$ it follows that,
$$Hom_{M^w}(r_{M(w),M}(\sigma),\mathbb{C})\cong \left(\otimes_{i\in\mathfrak{F}\;:\;\epsilon_w(i)=i} Hom_{H_{m_i}}(\sigma_i,\mathbb{C})\right)\otimes $$
$$\otimes \left(\otimes_{i\in\mathfrak{F}\;:\;\epsilon_w(i)\neq i} Hom_{\left\{(g^\tau,g)\in G_{m_i}\times G_{m_{\epsilon_w(i)}}\right\}}(\sigma_i\otimes \sigma_{\epsilon_w(i)},\mathbb{C}) \right).$$
The second term of the tensor product is evidently built from spaces of invariant pairings between $\sigma_i^\tau$ and $\sigma_{\epsilon_w(i)}$. 
\end{proof}

Finally we will need the following simple fact about induction and distinguished representations.

\begin{lem}\label{mult-functional}(\cite[Proposition 26]{fli92} or \cite[Lemma 6.4]{off})

Suppose that $M=M_{(m_1,\ldots,m_t)}< G_n$ is a standard Levi subgroup. Let $\sigma_i$ be a distinguished admissible representation of $G_{m_i}$, for all $1\leq i\leq t$. Then, $\pi= \sigma_1 \times\cdots\times \sigma_t$ is distinguished.
\end{lem}

\subsection{Blanc-Delorme theory}
For the standard Levi subgroup \\  $M=M_{(m_1,\ldots,m_t)}< G_n$, let us define the complex algebraic variety $X$ of unramified characters of $M$. That is, $X$ consists of characters of the form $\nu^\lambda:=\nu^{\lambda_1}\otimes\cdots \otimes\nu^{\lambda_t}$, where $\nu$ is the character (of any $G_{m_i}$) given by the formula $\nu(g) = |\det(g)|_E$, and $\lambda=(\lambda_1,\ldots,\lambda_t)\in\mathbb{C}^t$. The natural covering map $\mathbb{C}^t\to X$ equips $X$ with a structure of an affine algebraic variety isomorphic to $\left(\mathbb{C}^\times\right)^t$. As complex functions, the regular functions on $X$ composed with the covering map are polynomials in the variables $q^{\pm \lambda_1},\ldots, q^{\pm \lambda_t}$, where $q$ is the size of the residual field of $E$.
\par Let $\sigma$ be an admissible representation of $M$, and $(\pi,V)$ the representation of $G_n$ parabolically induced from $\sigma$ as before. It is possible to see $\pi$ as one element of a family of representations parametrized by unramified characters in $X$. Namely, for all $\chi\in X$, the representation $\pi_\chi:= ind^{G_n}_P (\chi\sigma)$ can be realized on the same space $V$ (see \cite[1.3]{FLO} for the precise construction), making $\pi_\chi(g)$ an analytic family of operators, for each $g\in G_n$. We will omit the description of such a realization, since it will not be of relevance here. 
\par Now, suppose that $m_i=m_{t+1-i}$, for all $1\leq i\leq t$. In this setting, we would like to exploit the theory developed in \cite{BD} to produce a non-zero $H_n$-invariant functional on $\pi$, under suitable conditions.
\par 
Let $\beta\in W_2[M]$ be the $M$-admissible element that is given by $\epsilon_\beta(i) = t+1-i$, for all $i\in \mathfrak{J}_M=\{1,\ldots,t\}$. Note, that $\eta_\beta H_n\eta_\beta^{-1}$ is the fixed point subgroup of the involution $\theta=\theta_\beta$ on $G_n$ given by $\theta(g) = \xi_\beta^{-1}g^\tau \xi_\beta$, where $\xi_\beta\in G_n$ is the permutation matrix corresponding to $\beta$. The subgroup $\theta(P)$ is then the opposite parabolic to $P$ relative to the $\theta$-stable maximal torus $T$ (the fixed diagonal torus).
\par Noting the action of $\theta$ on the characters of $G_n$, we define $X_\theta\subset X$ as the connected component of the identity character inside the affine variety of $\theta$-anti-invariants of $X$. Explicitly, 
$$X_\theta= \left\{\nu^{\lambda_1}\otimes\cdots\otimes \nu^{\lambda_t}\in X\::\: \forall i,\,\lambda_i = -\lambda_{t+1-i}\right\}\cong \left(\mathbb{C}^\times\right)^{\lfloor t/2 \rfloor}.$$
\par Suppose now that $\ell$ is a non-zero $H_n$-invariant functional on $\mathcal{V}_\beta(\pi)$. Recalling \cite[Lemma 6.4]{FLO} as before, this gives the existence of a non-zero functional $\widetilde{\ell}$ on the space of $\sigma$, which is invariant under the action of $M^\beta = M \cap G_n^\theta$. Now, when such $\widetilde{\ell}$ is put in the setting of \cite[Theorem 2.8]{BD}, we conclude the following statement. There is a regular function $r$ on $X_\theta$, such that for each $\chi\in X_\theta$ with $r(\chi)\neq 0$ there is a functional $0\neq J(\ell,\chi)\in V^\ast$ which is invariant under the action of $\pi_\chi|_{G_n^\theta}$. Moreover, for every $\phi\in V$, the function $\chi \mapsto r(\chi)J(\ell,\chi)(\phi)$ can be prolonged to a regular function on $X_\theta$.
\par This family of functionals can be used to construct a single non-zero functional on the original representation $\pi$. We summarize it in the following statement.
\begin{prop}\label{B-D}
Let $M=M_{(m_1,\ldots,m_t)}< G_n$ be a standard Levi subgroup, with $m_i=m_{t+1-i}$, for all $1\leq i\leq t$. Let $\pi$ be a representation of $G_n$ parabolically induced from $M$. Let $\beta\in W_2[M]$ be the $M$-admissible element that is given by $\epsilon_\beta(i) = t+1-i$, for all $i\in \mathfrak{J}_M=\{1,\ldots,t\}$. 
\par If $\mathcal{V}_\beta(\pi)$ has a non-zero $H_n$-invariant functional, then $\pi$ is distinguished.
\end{prop}
\begin{proof}
The analytic continuation of $L(\chi):=r(\chi)J(\ell,\chi)$ to $\chi=1$ (the trivial character) defines a functional on $V$ invariant under $\pi(\eta_\beta H_n\eta_\beta^{-1})$. Fix an (affine) curve $Y\subset X_\theta$ such that $1\in Y$, and $r|_Y$ is a non-zero function. It follows that $S =\{L(\chi)(\phi)|_Y\::\;\phi\in V\}$ is a collection of non-zero regular functions on $Y$. Let $p(\chi)$ be a regular function on $Y$ whose order of vanishing at $1$ equals to the minimum of these orders for $S$. Then, $\overline{L}= p(\chi)^{-1}L(\chi)|_{\chi=1}$ gives a non-zero functional. Hence, $0\neq \overline{L}\circ \pi(\eta_\beta)\in V^\ast$ is $H_n$-invariant.

\end{proof}

\section{Distinction of standard modules}
\subsection{Notations}
Denote by $\Pi_{G_n}$ the set of isomorphism classes of irreducible admissible representations of $G_n$. Given integers $a\leq b$, such that $b-a+1$ divides $n$, and a supercuspidal $\rho\in \Pi_{G_{n/(b-a+1)}}$, there is a unique irreducible quotient representation of $\nu^a\rho\times\nu^{a+1}\rho\times\cdots \times \nu^b\rho$. We denote this representation by $\Delta(\rho,a,b)$, and call it a \textit{segment}. These are exactly the essentially square-integrable representations in $\Pi_{G_n}$. We say that a segment $\Delta_1\in\Pi_{G_{n_1}}$ \textit{precedes} another segment $\Delta_2\in \Pi_{G_{n_2}}$ if $\Delta_1\cong \Delta(\rho,a_1, b_1)$ and $\Delta_2\cong \Delta(\rho,a_2,b_2)$ for some supercuspidal $\rho$ and integers with $a_1<a_2\leq b_1+1$ and $b_1<b_2$. A representation that is induced from two segments $\Delta_1\times \Delta_2$ is irreducible, if and only if, none of the segments precedes the other (\cite[9.7]{zel}). Also, we have $\Delta_1\times \Delta_2\cong \Delta_2\times \Delta_1$ when it is irreducible.
\par Denote by $c_\pi$ the central character of $\pi\in \Pi_{G_n}$. For all $g\in E^\times$, $|c_\pi(g)| = |g|_E^r$ for some $r\in \mathbb{R}$. We will call $r$ the real exponent of $\pi$ and denote it by $\Re(\pi)=r$. Clearly, $\Re(\nu\pi) = n +\Re(\pi)$. Also, if $\pi$ is a subquotient of $\pi_1\times \pi_2$, then $\Re(\pi)=\Re(\pi_1)+\Re(\pi_2)$. Together with the fact that $\nu^k\Delta(\rho,a,b)\cong\Delta(\rho,a+k,b+k)$ for any integer $k$, it is easy to see that if a segment $\Delta_1$ precedes $\Delta_2$, then $\Re(\Delta_1)<\Re(\Delta_2)$. 
\par The normalized Jacquet module of segments has a clear description (\cite[9.5]{zel}). Suppose that $\Delta=\Delta(\rho,a,b)\in \Pi_{G_n}$ with $\rho\in \Pi_{G_d}$. In case $d$ divides all $m_i$'s, we have $r_{M_{(m_1,\ldots,m_t)},G_n}(\Delta) = \Delta(\rho,a_1,b_1)\otimes\cdots\otimes \Delta(\rho,a_t,b_t)$, where $b_1=b$, $d(b_i-a_i+1)=m_i$, and $b_{i+1} =a_i-1$, for all $i$. Otherwise, the Jacquet module is the zero representation.
\par A representation of $G_n$ is called a \textit{standard module} if it is parabolically induced from a representation $\kappa\nu^\lambda$ of a standard Levi subgroup $M=M_{(m_1,\ldots,m_t)}< G_n$, where $\kappa$ is an irreducible tempered representation, and $\lambda=(\lambda_1,\ldots,\lambda_t)\in \mathbb{R}^t$, with $\lambda_1>\lambda_2>\ldots>\lambda_t$. Denote by $\mathfrak{S}_{G_n}$ the set of isomorphism classes of standard modules of $G_n$. Each element of it can be described by the triple $(M,\kappa,\lambda)$ known as the Langlands data. The Langlands classification for $G_n$ (proved in \cite{silber}) can be formulated as a bijection $\Sigma:\Pi_{G_n}\to \mathfrak{S}_{G_n}$ (or between $\Pi_{G_n}$ and triples of Langlands data), which satisfies the property that each $\pi\in \Pi_{G_n}$ is the unique irreducible quotient of $\Sigma(\pi)$. It is known that $\Sigma(\pi)=\pi$, if and only if, $\pi$ is generic. We will slightly abuse notation by using $\Sigma$ as a notation for a given element in $\mathfrak{S}_{G_n}$ as well.
\par In fact, our treatment of standard modules will not be focused on the above definition, but rather on the following well-known description of a standard module in terms of segments. 

\begin{prop}
Every $\Sigma\in \mathfrak{S}_{G_n}$ can be realized as an induced representation of the form $\Delta_1\times\cdots\times \Delta_t$, where each $\Delta_i$ is a segment, such that $\Delta_i$ does not precede $\Delta_j$ whenever $i<j$.
\end{prop}
\begin{proof}
Suppose that $\Sigma$ is parabolically induced from $\kappa_1\nu^{\lambda_1}\otimes\cdots \otimes \kappa_t\nu^{\lambda_t}$, as in the definition. It is known that a tempered representation $\kappa_i\in \Pi_{G_{m_i}}$ can be realized as an induced representation of the form $\Delta_{i,1}\times\cdots\times \Delta_{i,r_i}$, where $\{\Delta_{i,j}\}_j$ is a uniquely defined multiset\footnote{We will use this terminology to refer to a finite tuple of objects whose order is immaterial.} of segments, all of which have real exponent $0$. Bearing in mind that $\nu^{\lambda_i}\Delta(\rho,a,b)\cong\Delta(\nu^{\lambda_i}\rho,a,b)$ and the transitivity property of parabolic induction, we see that $\Sigma$ can be realized as $\Delta_1\times\cdots\times \Delta_t$, where each $\Delta_i$ is a segment, such that $\Re(\Delta_1)<\ldots<\Re(\Delta_t)$.
\end{proof}

We will call an induced representation as described in the above proposition a \textit{rang\'{e}e module}. Each rang\'{e}e module is a realization of a standard module. Two rang\'{e}e modules are isomorphic, if and only if, they differ by a permutation on their defining segments. Moreover, it is true that each multiset of segments can be reordered in such a way that their multiplication would give a rang\'{e}e module. Thus, $\mathfrak{S}_{G_n}$ is in bijection with multisets of segments. 
\par If $\mathcal{S}$ is a rang\'{e}e module realization of the standard module $\Sigma\in\mathfrak{S}_{G_n}$, we will write $\overline{\mathcal{S}}=\Sigma$. Yet, it will be useful to make a distinction between an element $\Sigma\in \mathfrak{S}_{G_n}$ and any of its concrete realizations $\mathcal{S}$ as a representation induced from a specified ordering of segments. 
\par For a supercuspidal $\rho\in \Pi_{G_k}$, consider the collection $[\rho]= \{\nu^l\rho\,:\,l\in\mathbb{Z}\}$. Given a rang\'{e}e module $\mathcal{S} = \Delta_1\times\cdots\times\Delta_t$, we set $\mathcal{S}_{[\rho]} = \Delta_{i_1}\times\cdots\times \Delta_{i_s}$, where $1\leq i_1< \ldots< i_s\leq t$ are the indices for which $\Delta_{i_j}\cong \Delta(\rho, a,b)$ for some $a,b$. Let us set $\mathcal{S}_{[\rho]}=1$ if there are no such indices. In these terms we always have a canonical (up to permutation) decomposition $\mathcal{S} \cong \mathcal{S}_{[\rho_1]}\times\cdots \times \mathcal{S}_{[\rho_l]}$ for some supercuspidals $\rho_1,\ldots,\rho_l$, such that $[\rho_i]$ and $[\rho_j]$ are disjoint for distinct $i,j$. Clearly, if $\overline{\mathcal{S}}= \Sigma$, then $\Sigma_{[\rho]}:= \overline{\mathcal{S}_{[\rho]}}$ is well-defined.
\par We will say that a rang\'{e}e module $\mathcal{S}$ is \textit{right-ordered}, if $\mathcal{S} = \mathcal{S}_{[\rho_1]}\times\cdots \times \mathcal{S}_{[\rho_l]}$ and for each $i$, $\mathcal{S}_{[\rho_i]} = \Delta(\rho_i, a^i_1,b^i_1)\times \cdots \times \Delta(\rho_i, a^i_{t_i}, b^i_{t_i})$ with $b^i_1\geq b^i_2\geq \cdots \geq b^i_{t_i}$, for all $i$. It is easily seen that each $\Sigma\in \mathfrak{S}_{G_n}$ has a (possibly non-unique) right-ordered rang\'{e}e module realization.
\par For $\pi\in \Pi_{G_n}$ we set $\pi_{[\rho_1]},\ldots,\pi_{[\rho_l]}$ to be the irreducible representations of the corresponding groups, such that $\Sigma(\pi_{[\rho_i]})= \Sigma_{[\rho_i]}$, for all $1\leq i\leq l$. This provides a decomposition of the form $\pi\cong \pi_{[\rho_1]}\times\cdots\times \pi_{[\rho_l]}$. The elements of the decomposition are sometimes called the pure components of $\pi$. When $\pi=\pi_{[\rho]}$, we will say that $\pi$ is of pure type $[\rho]$. Similar notation will also be used for standard modules.

\subsection{Jacquet's conjecture - First implication}
Let us recall the results of \cite[Propositions 11,12]{Fli91} which state that a representation $\pi\in \Pi_{G_n}$ has at most one non-zero $H_n$-invariant functional up to a scalar. If indeed $\pi$ is distinguished, then we must have $\pi^\tau\cong \pi^\vee$. As mentioned, the converse claim is not always true. Yet, let us investigate the standard module $\Sigma(\pi)$ of $\pi\in\Pi_{G_n}$ which satisfies $\pi^\tau\cong \pi^\vee$.
\par It is easy to see that segments satisfy $\Delta(\rho, a,b)^\tau\cong \Delta(\rho^\tau, a,b)$. It is also known (\cite[9.4]{zel}) that $\Delta(\rho,a,b)^\vee \cong \Delta(\rho^\vee, -b,-a)$. It then follows that a segment $\Delta_1$ precedes another segment $\Delta_2$, if and only if, $\Delta_1^\tau$ precedes $\Delta_2^\tau$, if and only if, $\Delta_2^\vee$ precedes $\Delta_1^\vee$. Thus, for a rang\'{e}e module $\mathcal{S} = \Delta_1\times\cdots\times \Delta_t$, both the representations $\mathcal{S}^\tau\cong \Delta_1^\tau\times\cdots\times \Delta_t^\tau$ and $\Delta_t^\vee\times\cdots\times \Delta_1^\vee$ will be rang\'{e}e modules. If $\overline{\mathcal{S}}=\Sigma$, let us denote by $\Sigma^\tau$ the isomorphism class of the former, and by $\Sigma^\ast$ the class of the latter. Let us remark though, that $\mathcal{S}^\vee$ is generally not a standard module, hence, $\Sigma^\ast$ must not be confused with the contragredient representation to $\Sigma$. 

\begin{prop}\label{obv}
Any $\pi\in \Pi_{G_n}$ satisfies $\Sigma(\pi^\tau) = \Sigma(\pi)^\tau$ and $\Sigma(\pi^\vee)= \Sigma(\pi)^\ast$. In particular, $\pi^\tau\cong \pi^\vee$ holds, if and only if, $\Sigma(\pi)^\ast= \Sigma(\pi)^\tau$.
\end{prop}

\begin{proof}
Since the Galois automorphism is obviously an exact functor, $\pi^\tau$ is the irreducible quotient of $\Sigma(\pi)^\tau$ and the first equality must hold. The second equality is proved, for example, in \cite[Proposition 5.6]{Tad}.
\end{proof}

\par Our first mission is to show that distinction of $\Sigma(\pi)$, for $\pi\in\Pi_{G_n}$, already imposes the condition $\pi^\tau\cong \pi^\vee$. For that we will need the following key lemma.
\begin{lem}\label{key-lem}
Suppose $\mathcal{S}$ is a right-ordered rang\'{e}e module realization of a standard module of $G_n$ that is induced from segments on a Levi subgroup $M$. Then, for every non-$M$-admissible element $w\in W_2[M]$, the space of $H_n$-invariant functionals on the representation $\mathcal{V}_w(\mathcal{S})$ is zero.
\end{lem}
\begin{proof}
Assume the contrary, that is, $w\in W_2[M]$ is a non-admissible element such that $\mathcal{V}_w(\mathcal{S})$ has a non-zero $H_n$-invariant functional $\ell$. 
\par Let us write $\delta=\otimes_{i\in \mathfrak{J}_M} \delta_i$ for the $M$-representation from which $\mathcal{S}$ is induced (each $\delta_i$ is a segment). We can also write $$r_{M(w),M}(\delta) = \otimes_{(i,j)\in\mathfrak{J}_{M(w),M}} \delta_{i,j},$$ where $\delta_{i,1}\otimes \ldots \otimes \delta_{i,s_i}$ is the Jacquet module of $\delta_i$ as a representation of the corresponding Levi subgroup of $G_{m_i}$. From the formula for Jacquet modules of segments we know that each $\delta_{i,j}$ must be a segment or the zero representation. In particular, if $\delta_{i,j}$ is distinguished, then it must satisfy $\delta_{i,j}^\vee\cong \delta_{i,j}^\tau$. Combining the last fact, Lemma \ref{geom-functional}, and the existence of $\ell$, we conclude that the $\delta_{i,j}$'s must all be non-zero representations and that $\delta_{\epsilon_w(i,j)}^\vee\cong \delta_{i,j}^\tau$ holds for all $(i,j)\in \mathfrak{J}_{M(w),M}$.
\par Since $w$ is non-$M$-admissible, $M(w)$ is strictly contained in $M$, which means there exists $i_0\in \mathfrak{J}_M$ with $s_{i_0}>1$. Let us assume $i_0$ is the minimal such index. Then, $\delta_{i_0} = \Delta(\rho,a,b)$ for some supercuspidal $\rho$ and integers $a<b$, $\delta_{i_0,1} = \Delta(\rho, d+1,b)$ and $\delta_{i_0,2} = \Delta(\rho,c,d)$ for some integers $a\leq c\leq d< b$. Suppose that $\epsilon_w(i_0,1) = (i_1,j_1)$ and $\epsilon_w(i_0,2) = (i_2, j_2)$. We know that $\delta_{i_1,j_1}\cong \left(\delta_{i_0,1}^\tau\right)^\vee\cong \Delta\left(\left(\rho^\tau\right)^\vee, -b,-d-1\right)$, and similarly $\delta_{i_2,j_2}\cong \Delta(\left(\rho^\tau\right)^\vee, -d,-c)$. We also know that $i_1<i_ 2$. Recalling that $\mathcal{S}$ was right-ordered, this must mean that $\delta_{i_1}\cong \Delta(\left(\rho^\tau\right)^\vee, a', b')$ for some $a'\leq -b$ and $-c\leq b'$. Now, since $-d-1<-c$, we deduce that $j_1>1$ and that $\delta_{i_1,1}\cong \Delta(\left(\rho^\tau\right)^\vee, e,b')$ for some $-d\leq e\leq b'$. But, that means $\epsilon_w(i_1,1)=(i_3,j_3)$ for some $i_3<i_0$. From minimality of $i_0$, we must have $j_3=1$ and $\delta_{i_3} = \delta_{i_3,j_3} \cong \left(\delta_{i_1,1}^\tau\right)^\vee \cong \Delta(\rho,-b',-e)$. Finally, notice that $-e< b$, which is a contradiction to $\mathcal{S}$ being right-ordered.
\end{proof}

Let $\mathcal{S}=\Delta_1\times\cdots\times \Delta_t$ be a right-ordered rang\'{e}e module. Let $\ell$ be a non-zero $H_n$-invariant functional on $\mathcal{S}$. When filtering the representation space of $\mathcal{S}$ as an induced representation from segments on a Levi subgroup $M$, $\ell$ must induce a non-zero $H_n$-invariant functional on $\mathcal{V}_w(\mathcal{S})$, for some $w\in W_2[M]$. By the above lemma, $w$ must be $M$-admissible, hence, $\epsilon_w$ is a permutation on $\mathfrak{J}_M$. Recalling Lemma \ref{geom-functional}, we see that $\Delta^\tau_{\epsilon_w(i)}\cong \Delta_i^\vee$ for all $i\in \mathfrak{J}_M$ with $\epsilon_w(i)\neq i$ and that $\Delta_i$ is distinguished for $i\in\mathfrak{J}_M$ such that $\epsilon_w(i)=i$. This analysis has the following corollaries. 

\begin{prop}\label{first-dir}
If $\mathcal{S}= \Delta_1\times \cdots \Delta_t$ is any rang\'{e}e module realization of a distinguished standard module of $G_n$, then $\Delta^\tau_{\epsilon(i)}\cong \Delta_i^\vee$ for some involution $\epsilon$ on $\{1,\ldots,t\}$. In particular, $\overline{\mathcal{S}}^\ast =\overline{\mathcal{S}}^\tau$.
\par Moreover, when $\epsilon(i)=i$, the segment $\Delta_i$ is distinguished.
\end{prop}

\begin{proof}
There is a permutation $\alpha$ of $\{1,\ldots,t\}$ such that $\mathcal{S}_\alpha= \Delta_{\alpha(1)}\times\cdots\times \Delta_{\alpha(t)}$ is a right-ordered rang\'{e}e module, and $\overline{\mathcal{S}_\alpha}=\overline{\mathcal{S}}$. Recalling again that a distinguished segment $\Delta_i\in \Pi_{G_{m_i}}$  must satisfy $\Delta_i^\tau\cong \Delta_i^\vee$, we see that $\epsilon := \alpha^{-1}\epsilon_w\alpha$ would fill the requirements of the statement.
\par The existence of such $\epsilon$ also shows the multisets of segments defining both $\overline{\mathcal{S}}^\ast$ and $\overline{\mathcal{S}}^\tau$ are the same.
\end{proof}
\begin{rem}
The above can be seen as a generalization of the main theorem of \cite{matr-gen}, where the case of a generic irreducible representation was handled, that is, when $\overline{\mathcal{S}} = \Sigma(\pi)=\pi\in \Pi_{G_n}$.
\end{rem}

\begin{cor}\label{cor-firdir}
If $\Sigma(\pi)$ is distinguished for $\pi\in \Pi_{G_n}$, then $\pi^\vee\cong \pi^\tau$.
\end{cor}

\begin{prop}\label{mult-one-prop}
Suppose that $\mathcal{S}=\Delta_1\times\cdots\times \Delta_t$ is a rang\'{e}e module realization of a standard module of $G_n$, such that $\Delta_i\not\cong \Delta_j$ for all distinct $i,j$. Then, $\dim Hom_{H_n}(\mathcal{S}, \mathbb{C})\leq 1$.

\end{prop}
\begin{proof}
As before, we can assume $\mathcal{S}$ is right-ordered and parabolically induced from segments on a standard Levi subgroup $M<G_n$. Suppose that $\mathcal{S}$ is distinguished. We have seen that as a consequence $\mathcal{V}_w(\mathcal{S})$ is distinguished for some $M$-admissible $w\in W_2[M]$. That forces $\Delta^\tau_{\epsilon_w(i)}\cong \Delta_i^\vee$ for all $i$, as in the proof of Proposition \ref{first-dir}. Now, if $\mathcal{V}_{w'}(\mathcal{S})$ had been distinguished for some other $M$-admissible $w\neq w'\in W_2[M]$, it would have also imposed the condition $\Delta^\tau_{\epsilon_{w'}(i)}\cong \Delta_i^\vee$. But, that condition cannot hold for two different involutions because of our assumption. Hence, $\mathcal{V}_w(\mathcal{S})$ is the only distinguished geometric subquotient of $\mathcal{S}$. Yet, since segments are irreducible, by the general multiplicity-one theorem $\dim Hom_{H_{m_i}}(\Delta_i,\mathbb{C})\leq 1$ (where $\Delta_i\in \Pi_{G_{m_i}}$). Together with Lemma \ref{geom-functional}, it implies that $\dim Hom_{H_n}(\mathcal{V}_w(\mathcal{S}),\mathbb{C})\leq 1$. The validity of the statement follows easily.
\end{proof}

\subsection{Jacquet's conjecture - Converse implication}
We treat the converse problem, that is, what can be said about the distinction properties of a general standard module $\Sigma$ which satisfies $\Sigma^\ast=\Sigma^\tau$. 
\par Let $\eta$ be the quadratic character of $F^\times$ associated to the extension $E/F$. Let $\chi$ be any extension of $\eta$ to $E^\times$. We also denote by $\eta$ and $\chi$ the corresponding characters of $H_m$ and $G_m$, for any $m$, obtained by composition with the determinant maps $G_m\to G_1$, $H_m\to H_1$. Since $\chi$ is trivial on the group of norms (for the extension $E/F$) of $E^\times$, we have $\chi^\tau = \chi^{-1}$.

The representation $\chi\pi$ is distinguished, if and only, $\pi$ is $\eta$-distinguished. Note, that $\chi(\pi_1\times\pi_2)\cong (\chi\pi_1)\times(\chi \pi_2)$, and that $\chi \Delta(\rho,a,b)\cong \Delta(\chi \rho,a,b)$ for all segments. In particular, $\mathfrak{S}_{G_n}$ is closed under tensoring with $\chi$. In this context we note the following obvious corollary of Lemma \ref{mult-functional}.
\begin{cor}\label{eta-cor}
Suppose that $M=M_{(m_1,\ldots,m_t)}< G_n$ is a standard Levi subgroup. Let $\sigma_i$ be a $\eta$-distinguished admissible representation of $G_{m_i}$, for all $1\leq i\leq t$. Then, $\pi= \sigma_1 \times\cdots\times \sigma_t$ is $\eta$-distinguished.
\end{cor}

Let us recall what is known about distinction of segments. We will recast the accumulated results of \cite{Kab},\cite{matr-early} and \cite{AKT} on the issue into a unified notation. 
\par Note, that $\Re((\pi^\tau)^\vee)=-\Re(\pi)$ for all $\pi\in \Pi_{G_n}$.  Now, when $\rho\in \Pi_{G_n}$ is a supercuspidal representation satisfying $[(\rho^\tau)^\vee] = [\rho]$, we have $\Re(\rho)= nr + \Re((\rho^\tau)^\vee)= nr - \Re(\rho)$ for some integer $r$. Thus, $\frac{\Re(\rho)}n$ must be half-integer.

\begin{prop}\label{cusp-known}
Let $\rho\in \Pi_{G_n}$ be a supercuspidal representation satisfying $(\rho^\tau)^\vee = \rho$. Then, there is a bit $\gamma(\rho)\in\{0,1\}$, such that $\rho$ is $\eta^{\gamma(\rho)}$-distinguished.
\end{prop}
\begin{proof}
It follows from the assumption that $\Re(\rho)=0$. As a supercuspidal representation with a unitary central character, $\rho$ is square-integrable. The statement then follows essentially from the main theorem of \cite{Kab}. We only remark why the requirement on the central character in Kable's result is superfluous with our formulation. Indeed, the proof of \cite[Theorem 7]{Kab} shows that the local Asai L-function of one of the representations $\rho$ and $\chi\rho$ has a pole at $0$. Applying \cite[Theorem 4]{Kab} is then enough to finish the argument.
\end{proof}

\begin{prop}\label{known}
Let $\rho\in \Pi_{G_n}$ be a supercuspidal representation, and $\Delta = \Delta(\rho,a,b)$ a segment. Then, 
\begin{enumerate}
\item The identity $\Delta\cong (\Delta^\tau)^\vee$ holds, if and only if, $[(\rho^\tau)^\vee] = [\rho]$ and $\Re(\Delta)=0$.
\item Suppose that $[(\rho^\tau)^\vee] = [\rho]$ holds. Then, there is an invariant bit $\gamma(\rho)= \gamma\left([\rho]\right)\in\{0,1\}$ extending the previous definition of $\gamma$ for unitarizable supercuspidals, such that if $\Delta\cong (\Delta^\tau)^\vee$ holds, then the segment $\Delta$ is $\eta^{\gamma(\rho)}$-distinguished, and is not $\eta^{\gamma(\rho)+1}$-distinguished.
\end{enumerate}
\end{prop}
\begin{proof}
1. If $\Delta\cong (\Delta^\tau)^\vee$ holds, then $[(\rho^\tau)^\vee] = [\rho]$ and $\Re(\Delta)=0$ are immediate. Conversely, suppose that $[(\rho^\tau)^\vee] = [\rho]$ holds. Note, that 
$$\Re(\Delta)= \frac{b-a+1}2\left(na+nb+2\Re(\rho)\right)$$
as a sum of an arithmetic progression. Thus, when $\Re(\Delta)=0$, we have $a+b=-2\Re(\rho)/n$. Since $\rho$ is the only element of $[\rho]$ with real exponent $\Re(\rho)$, we must have $\nu^{2\Re(\rho)/n}(\rho^\tau)^\vee\cong \rho$. From this, $\Delta((\rho^\tau)^\vee,-b,-a)\cong \Delta(\rho,a,b)$ easily follows.
\par 2. We denote $\rho_0 = \nu^{-\Re(\rho)/n}\rho$ and $\gamma_0=\gamma(\rho_0)$, as defined in Proposition \ref{cusp-known}. When $\Delta\cong (\Delta^\tau)^\vee$ holds, \cite[Corollary 4.2]{matr-early} states that $\Delta$ is $\eta^{\gamma_0+a-b}$-distinguished. Yet, we have observed above that in this the parity of $a-b$ is the same as $2\Re(\rho)/n$. Thus, 
$$\gamma(\rho):= \gamma_0 + 2\Re(\rho)/n \;(\mod 2\,)$$ 
will satisfy the condition in the statement. This definition of $\gamma$ is also easily seen to be an invariant of $[\rho]$. Finally, the fact that $\Delta$ cannot be both distinguished and $\eta$-distinguished is proved in \cite{AKT}.
\end{proof}

\begin{rem}
We clearly have $\gamma(\chi\rho) =1-\gamma(\rho)$. 
\par Also, if $[(\rho^\tau)^\vee] = [\rho]$, note that $((\nu^{1/2}\rho)^\tau)^\vee = \nu^{-1/2}(\rho^\tau)^\vee \in [\nu^{1/2}\rho]$. Thus, $\gamma(\nu^{1/2}\rho)$ is well-defined. Going through the definition of $\gamma$ in the above proof, one can deduce that $\gamma(\nu^{1/2}\rho) = 1- \gamma(\rho)$.
\end{rem}

The knowledge of the distinction properties of segments can be combined with the Blanc-Delorme method for producing functionals, and thus learning about the distinction of standard modules. Let us start with the treatment of a standard module of pure type.
\begin{lem}\label{sec-dir}
Suppose that $\Sigma = \Sigma_{[\rho]}= \left(\Sigma^\tau\right)^\ast\in \mathfrak{S}_{G_n}$ for a supercuspidal $\rho$. Then $\Sigma$ is $\eta^{\gamma(\rho)}$-distinguished. 
\par If the real exponents of all segments from which $\Sigma$ is constructed are non-zero, then $\Sigma$ is both distinguished and $\eta$-distinguished. 
\end{lem}
\begin{proof}
Let us write the multiset of segments which define $\Sigma$ as $\{\Delta_i\}_{i\in I}$. Let us partition this multiset as $I = I_{-}\cup I_0 \cup I_+$ according to whether $\Re (\Delta_i)$ is negative, zero or positive. Define $\mathcal{S}_0 = \times_{i\in I_0} \Delta_i$, with an arbitrary order of multiplication. Recall again, that if a segment $\Delta'$ precedes $\Delta^{''}$, we must have $\Re(\Delta')<\Re(\Delta^{''})$. Since $\Re(\Delta_i)=0$ for all $i\in I_0$, $\mathcal{S}_0$ is a rang\'{e}e module.
\par Clearly, the assumption on $\Sigma$ compels $[(\rho^\tau)^\vee] = [\rho]$ to hold. So, by Proposition \ref{known}, all $\{\Delta_i\}_{i\in I_0}$ are $\eta^{\gamma(\rho)}$-distinguished. By passing from $\Sigma$ to $\chi^{\gamma(\rho)} \Sigma$, we may assume they are in fact all distinguished while retaining the condition $\Sigma^\tau = \Sigma^\ast$. Thus, by Lemma \ref{mult-functional}, $\mathcal{S}_0$ is distinguished. Clearly when $\mathcal{S}_0=1$, what follows applies to both $\Sigma$ and $\chi \Sigma$. 
\par Let $\mathcal{S}_+ = \Delta^+_1\times\cdots\times \Delta^+_t$ be a rang\'{e}e module constructed from the segments $\{\Delta_i\}_{i\in I_+}$. By our assumption $\{\Delta_i\}_{i\in I}$ is closed under the operation $\Delta\mapsto (\Delta^\tau)^\vee$. Since this operation negates the real exponent of a segment, it actually induces a bijection $w:I_+\to I_-$ such that $\Delta_{w(i)}\cong (\Delta_i^\tau)^\vee$, for all $1\leq i\leq t$. If $\Delta^+_j = \Delta_i$, let us define $\Delta^-_j:= \Delta_{w(i)}$. 
Then, $\mathcal{S}_{-} = \Delta^{-}_t \times \cdots \times \Delta^-_1$ is a rang\'{e}e module.
\par Moreover, arguing as before about the real exponent of preceding segments, we see that 
$$\Delta^+_1\times\cdots\times \Delta^+_t \times \times_{i\in I_0} \Delta_i \times \Delta^{-}_t \times \cdots \times \Delta^-_1$$
is a rang\'{e}e module realization of $\Sigma$. In particular, it obviously means that $\Sigma$ can be realized as the induced representation 
$\pi:=\Delta^+_1\times\cdots\times \Delta^+_t \times \mathcal{S}_0 \times \Delta^{-}_t \times \cdots \times \Delta^-_1$.
Since $\Delta^+_i\cong ((\Delta^-_i)^\tau)^\vee$ for all $1\leq i\leq t$ and $\mathcal{S}_0$ is distinguished, by Lemma \ref{geom-functional} $\mathcal{V}_\beta(\pi)$ is distinguished, where $\beta$ is as described in Proposition \ref{B-D}. So, by that proposition $\pi$ must be distinguished as well.
\end{proof}

In general, we can now formulate the analogue of the Jacquet's conjecture on the level of standard modules.
\begin{thm}\label{jac-mod}
Let $\pi\in \Pi_{G_n}$ be such that $\pi^\vee \cong \pi^\tau$. Then, there is a decomposition $\pi\cong \pi_1\times \pi_2\times \pi_3$, where $\pi_i\in \Pi_{G_{n_i}}$ for $i=1,2,3$, $\Sigma(\pi) = \Sigma(\pi_1)\times \Sigma(\pi_2)\times \Sigma(\pi_3)$, $\Sigma(\pi_1)$ is both distinguished and $\eta$-distinguished, $\Sigma(\pi_2)$ is distinguished but not $\eta$-distinguished, and $\Sigma(\pi_3)$ is not distinguished but $\eta$-distinguished. Each of $\pi_1, \pi_2, \pi_3$ may be missing from the decomposition.

\end{thm}
\begin{proof}
Let $\Sigma(\pi) = \Sigma(\pi_{[\rho_1]})\times \cdots\times \Sigma(\pi_{[\rho_t]})$ be the canonical decomposition to standard modules of the pure components of $\pi$. 
By Proposition \ref{obv}, $\Sigma(\pi)^\ast = \Sigma(\pi)^\tau$. Hence, there is an involution $w$ on $\{1,\ldots,t\}$, such that $\left[\rho_{w(i)}^\tau\right] = [\rho_i^\vee]$, and $\Sigma(\pi_{[\rho_{w(i)}]})^\ast = \Sigma(\pi_{[\rho_i]})^\tau$. Since a change in the order of the pure components is of no effect, we can assume that there is $0\leq r<t$ such that $w(2i)=2i-1$ for all $1\leq i\leq r$, and $w(i)=i$ for all $2r<i\leq t$.
\par Fix $1\leq i\leq r$. Let $\Delta_1\times\ldots \times \Delta_k$ be a realization of $\Sigma(\pi_{[\rho_{2i}]})$. Then, 
$$\mathcal{S} = (\Delta_k^\tau)^\vee\times\cdots \times (\Delta_1^\tau)^\vee\times \Delta_1\times\cdots \times \Delta_k$$
must be a realization of $\Sigma(\pi_{[\rho_{2i-1}]})\times \Sigma(\pi_{[\rho_{2i}]})$. Yet, by Proposition \ref{B-D} and Lemma \ref{geom-functional}, $\mathcal{S}$ is distinguished. Furthermore, the fact that $\left(\left(\chi \Delta\right)^\tau\right)^\vee\cong \chi\left(\Delta^\tau\right)^\vee$ shows that $\chi \mathcal{S}$ is distinguished by the same argument.
\par By invoking Lemma \ref{mult-functional}, we see that $\Sigma'= \left(\Sigma(\pi_{[\rho_1]})\times\Sigma(\pi_{[\rho_2]})\right)\times\cdots\times\left( \Sigma(\pi_{[\rho_{2r-1}]})\times\Sigma(\pi_{[\rho_{2r}]})\right)$ is both distinguished and $\eta$-distinguished.
\par Finally, we have $\Sigma(\pi)= \Sigma'\times \Sigma(\pi_{[\rho_{2r+1}]})\times\cdots\times \Sigma(\pi_{[\rho_t]})$, where each $2r<i\leq t$ satisfies $\Sigma(\pi_{[\rho_{i}]})^\ast = \Sigma(\pi_{[\rho_i]})^\tau$. Thus, after a proper rearrangement of the indices, by Lemma \ref{sec-dir} there will be $2r+1\leq s_1\leq s_2\leq t$, such that $\Sigma(\pi_{[\rho_{i}]})$ is both distinguished and $\eta$-distinguished for all $2r+1\leq i\leq s_1$, distinguished for all $s_1\leq i\leq s_2$, and $\eta$-distinguished for all $s_1\leq i\leq t$. 
Put
$$\pi_1'=\times_{i=1}^{s_1}\pi_{[\rho_i]} \quad \pi_2'=\times_{i=s_1+1}^{s_2}\pi_{[\rho_i]}\quad \pi_3'=\times_{i=s_2+1}^{t}\pi_{[\rho_i]}.$$ The statement will then be partially satisfied by Lemma \ref{mult-functional} and Corollary \ref{eta-cor}, when setting $\pi_i=\pi_i'$. For the complete statement, we should switch to $\pi_1:=\pi_1'\times \pi_2'$ and omit $\pi_2$ from the decomposition, in case $\Sigma(\pi_2')$ happens to be $\eta$-distinguished as well. A similar switch can be done when $\pi_3'$ happens to be distinguished.
\end{proof}

Reading the above theorem, it may be tempting to conjecture that the irreducible representations $\pi_i$, $i=1,2,3$, should satisfy the same distinction properties as their respective standard modules. Yet, we will see that the class of ladder representations serves as a source of examples for $\pi\in \Pi_{G_n}$ of pure type, whose standard module is both distinguished and $\eta$-distinguished, while $\pi$ itself satisfies only one ``kind" of distinction.

\section{Distinction of ladder representations}

We will say that a rang\'{e}e module $\mathcal{S}= \Delta_1\times\cdots \times \Delta_t$ is of \textit{proper ladder type} if $\Delta_i$ precedes $\Delta_{i-1}$, for all $1< i\leq t$. Note, that for such $\mathcal{S}$, the standard module $\overline{\mathcal{S}}\in \mathfrak{S}_{G_n}$ has a unique rang\'{e}e module realization. Hence, we can safely say that $\overline{\mathcal{S}}$ is of proper ladder type. A representation $\pi\in \Pi_{G_n}$ is called a \textit{proper ladder} representation if $\Sigma(\pi)$ is of proper ladder type. In particular, proper ladder representations are of pure type.

In \cite{LM}, a \textit{ladder} representation was defined as $\pi\in \Pi_{G_n}$, for which $\Sigma(\pi)$ can be realized as $\Delta(\rho,a_1,b_1)\times\cdots\times \Delta(\rho,a_t,b_t)$, with $a_1>a_2>\cdots>a_t$ and $b_1>b_2>\cdots >b_t$. 

The following straightforward lemma shows that ladder representations are easily decomposed into proper ladder representations.

\begin{lem}\label{prop-decomp}
Let $\pi$ be a ladder representation. There are unique proper ladder representations $\pi_1,\ldots,\pi_k$ such that

$$ \Sigma(\pi) = \Sigma(\pi_1)\times\cdots\times\Sigma(\pi_k) $$
holds, and for which the defining segments of $\Sigma(\pi_i)$ do not precede those of $\Sigma(\pi_j)$, for distinct $i,j$. Moreover, $\pi\cong \pi_1\times\cdots\times\pi_k$.
\end{lem}

\begin{proof}
Suppose that $\pi$ is realized as $\Delta(\rho,a_1,b_1)\times\cdots\times \Delta(\rho,a_t,b_t)$. We write
$$i_1 = \min \{1\leq i \leq t\;:\; b_{i+1}< a_i-1\}\quad (b_{t+1}=-\infty),$$ 
and denote $\Sigma_1 = \Delta(\rho,a_1,b_1)\times\cdots\times \Delta(\rho,a_{i_1},b_{i_1})$. Next, we write
$$i_2 = \min \{i_1< i \leq t\;:\; b_{i+1}< a_i-1\},$$ 
and $\Sigma_2 = \Delta(\rho,a_{i_1+1},b_{i_1+1})\times\cdots\times \Delta(\rho,a_{i_2},b_{i_2})$. We continue like so until $i_k=t$. 

The irreducible representations $\pi_1,\ldots,\pi_k$ for which $\Sigma(\pi_i)=\Sigma_i$ clearly satisfy the first conditions of the statement. The equality $\pi\cong \pi_1\times\cdots\times\pi_k$ then follows from \cite[Proposition 8.5]{zel}.
\end{proof}

Suppose that $\Delta = \Delta(\rho,a,b)$ is a segment that precedes $\Delta'=\Delta(\rho,a',b')$. Then, $(\Delta\cup \Delta')\times (\Delta \cap \Delta')$ is a sub-representation of $\Delta'\times \Delta$, where $\Delta'\cup \Delta = \Delta(\rho, a,b')$ and $\Delta'\cap \Delta = \Delta(\rho, a',b)$. Moreover, if $\Sigma(\pi) = \Delta'\times \Delta$, then $\pi$ is given as $\Sigma(\pi)/\left((\Delta\cup \Delta')\times (\Delta \cap \Delta')\right)$. Such description of the maximal sub-representation of a standard module was generalized in \cite{LM} for standard modules of ladder representations.
\par Suppose that $\pi\in \Pi_{G_n}$ is a proper ladder representation, with $\mathcal{S}=\Delta_1\times\cdots\times \Delta_t$ realizing $\Sigma(\pi)$. For all $i=1,\ldots,t-1$, we define
$$\mathcal{S}_i = \Delta_1\times\cdots\times\Delta_{i-1} \times (\Delta_{i+1}\cup \Delta_i)\times (\Delta_{i+1} \cap \Delta_i)\times \Delta_{i+2}\times\cdots\times \Delta_t.$$
It is easy to check that the $\mathcal{S}_i$'s are all standard modules, which by exactness of parabolic induction can be embedded as sub-representations of $\mathcal{S}$. Let us denote $\Sigma_i =\overline{\mathcal{S}_i}$, and consider $\Sigma_1,\ldots, \Sigma_{t-1}\subset \Sigma(\pi)$ as sub-representations. The main theorem of \cite{LM} states that $\pi\cong \Sigma(\pi)/\left(\Sigma_1+\cdots+\Sigma_{t-1}\right)$. We would like to use this description to obtain invariant functionals on ladder representations.

\begin{thm}\label{ladder-thm}
Let $\pi = \pi_{[\rho]}\in \Pi_{G_n}$ be a proper ladder representation, with $\Sigma(\pi)= \Delta_1\times\cdots\times \Delta_t$. Then $\pi^\tau\cong \pi^\vee$ holds, if and only if, $\pi$ is $\eta^{\gamma(\rho)+t+1}$-distinguished.
\end{thm}

\begin{proof}
As mentioned before, when $\pi$ is distinguished, $\pi^\tau\cong \pi^\vee$ follows from \cite[Proposition 12]{Fli91}. In the same manner, when $\chi\pi$ is distinguished, $\chi^\tau \pi^\tau \cong \chi^{-1}\pi^\vee$ holds, but $\chi^\tau =\chi^{-1}$.

Conversely, suppose that $\pi^\tau\cong \pi^\vee$ holds. 

By Proposition \ref{obv} we have $\Sigma(\pi)^\ast=\Sigma(\pi)^\tau$. Hence, $(\Delta_t^\tau)^\vee\times\cdots \times (\Delta_1^\tau)^\vee\cong \Delta_1\times\cdots\Delta_t$. Since there is only one rang\'{e}e module realization of a standard module of proper ladder type, we must have $\Delta_{t+1-i} \cong \left(\Delta_i^\tau\right)^\vee$. In particular, $\Re(\Delta_i)=-\Re(\Delta_{t+1-i})$. The ladder condition also imposes $\Re(\Delta_1)>\Re(\Delta_2)>\ldots > \Re(\Delta_t)$. Thus, if $t$ is even, $\Re(\Delta_i)\neq 0$ for all $i$. Lemma \ref{sec-dir} then indicates that $\Sigma(\pi)$ is $\eta^{\gamma(\rho)+t+1}$-distinguished. The same lemma also gives the same conclusion when $t$ is odd ($\eta^{\gamma(\rho)+t+1}=\eta^{\gamma(\rho)}$).
\par We exhibited a non-zero $(H_n,\eta^{\gamma(\rho)+t+1})$-equivariant functional on $\Sigma(\pi)$. Now, we would like to show that it factors through the map $\Sigma(\pi)\mapsto \pi$. In other words, we like to show the functional vanishes on each $\Sigma_i,\:i=1,\ldots, t-1$. It is enough to show that these standard modules are not $\eta^{\gamma(\rho)+t+1}$-distinguished.
\par Assume the contrary, that is, $\Sigma'=\chi^{\gamma(\rho)+t+1}\Sigma_{i_0}$ is distinguished, for some $i_0$. Let $\Delta'_1\times\cdots\times \Delta'_t$ be a realization of $\Sigma'$, where $\Delta'_j= \chi^{\gamma(\rho)+t+1}\Delta_j$ for $1\leq j<i_0$ or $i_0+1<j\leq t$, and $\Delta'_{i_0}= \chi^{\gamma(\rho)+t+1}\left(\Delta_{i_0+1}\cup \Delta_{i_0}\right)$, $\Delta'_{i_0+1}= \chi^{\gamma(\rho)+t+1}\left(\Delta_{i_0+1}\cap \Delta_{i_0}\right)$. By Proposition \ref{first-dir} there is an involution $\epsilon$ on $\{1,\ldots,t\}$, such that $\left(\Delta'^\tau_{j}\right)^\vee\cong \Delta'_{\epsilon(j)}$.

\par Suppose that $t+1-i_0\not\in \{i_0,i_0+1\}$. Then, 
$$\Delta'_{\epsilon( t+1-i_0)}\cong \chi^{\gamma(\rho)+t+1} \left(\Delta_{t+1-i_0}^\tau\right)^\vee\cong \chi^{\gamma(\rho)+t+1}\Delta_{i_0}.$$
Yet, since $\Sigma(\pi)$ is of ladder type, it can be easily seen that none of $\Delta'_1,\ldots,\Delta'_t$ can be isomorphic to $\chi^{\gamma(\rho)+t+1}\Delta_{i_0}$. Thus, we must have $t+1-i_0\in \{i_0,i_0+1\}$. The same argument also shows that $t+1-(i_0+1)\in\{i_0,i_0+1\}$. In other words, $t$ must be even, and $i_0=t/2$.
\par Now, by repeating a similar argument we can see that if $\epsilon(i_0)\not\in \{i_0,i_0+1\}$, then $\Delta'_{i_0}\cong \chi^{\gamma(\rho)+t+1} \Delta_{t+1-\epsilon(i_0)}$. That would have meant $\Delta_{i_0+1}\cup \Delta_{i_0} \cong \Delta_{t+1-\epsilon(i_0)}$, which is again a contradiction, because $\Sigma(\pi)$ is of ladder type. So, $\epsilon(i_0)\in \{i_0,i_0+1\}$. Note, that $\left(\Delta'^\tau_{i_0}\right)^\vee\cong \Delta'_{i_0+1}$ cannot hold, because those are segments of different lengths. Hence, $\epsilon(i_0)=i_0$, and by Proposition \ref{first-dir} it means $\Delta_{i_0+1}\cup \Delta_{i_0}$ is $\eta^{\gamma(\rho)+t+1}$-distinguished. Recalling that $t$ is even, this gives a contradiction to Proposition \ref{known}.

\end{proof}

\begin{thm}\label{ladder-thm-2}
Let $\pi=\pi_{[\rho]}\in \Pi_{G_n}$ be a ladder representation satisfying 
$$\pi^\tau\cong \pi^\vee\;.$$

Suppose that $\pi \cong \pi_1\times\cdots\times \pi_k$, with each $\pi_i$ being a proper ladder representation, as in Lemma \ref{prop-decomp}. Suppose further that $\Sigma(\pi) =  \Delta_1\times\cdots\times \Delta_t$.

Then, in case $k$ is even, $\pi$ is both distinguished and $\eta$-distinguished.

In case $k$ is odd, $\pi$ is $\eta^{\gamma(\rho)+t+1}$-distinguished.

\end{thm}
\begin{proof}
Recall that $\Sigma(\pi) = \Sigma(\pi_1)\times\cdots\times \Sigma(\pi_k)$ and by Proposition \ref{obv},
$$ \Sigma(\pi_k)^\ast\times\cdots\times \Sigma(\pi_1)^\ast = \Sigma(\pi)^\ast = \Sigma(\pi)^\tau = \Sigma(\pi_1)^\tau\times\cdots\times \Sigma(\pi_k)^\tau.$$
The above equality shows two decompositions of a standard module of a ladder representation into standard modules of proper ladder type, each of which ordered by decreasing real exponent of their irreducible quotient. Thus, from uniqueness we have $\pi^\vee_i\cong \pi^\tau_{k+1-i}$, for all $1\leq i\leq k$. 

In case $k$ is even, we can use Proposition \ref{B-D} to produce the desired functionals on $\pi$, and on $\chi\pi$, by applying the same argument as in the end of the proof of Lemma \ref{sec-dir}.

In case $k$ is odd, this argument would still be valid if we knew that $\pi_{(k+1)/2}\cong (\pi_{(k+1)/2}^\tau)^\vee$ is $\eta^{\gamma(\rho)+t+1}$-distinguished. Indeed, by Theorem \ref{ladder-thm}, $\pi_{(k+1)/2}$ is $\eta^{\gamma(\rho)+t_0+1}$-distinguished, where $t_0$ is the number of segments in the construction of $\Sigma(\pi_{(k+1)/2})$. Note, that for any $i\neq (k+1)/2$, $\Sigma(\pi_i)$ and $\Sigma(\pi_{k+1-i})$ are distinct standard modules constructed from an equal number of segments. Thus, $t-t_0$ is an even number, and $\eta^{\gamma(\rho)+t+1}=\eta^{\gamma(\rho)+t_0+1}$.

\end{proof}


Finally, we want to complete the above result by showing that a ladder representation that is induced from an odd number of proper ladder representations, cannot be both distinguished and $\eta$-distinguished.

\begin{lem}\label{noboth-lem}
Suppose that $\mathcal{S}=\mathcal{S}_{[\rho]}$ is a rang\'{e}e module of pure type induced from $t$ segments. If $t$ is odd, then $\mathcal{S}$ is not $\eta^{\gamma(\rho)+1}$-distinguished.
\end{lem}
\begin{proof}
Assume that $\mathcal{S}':= \chi^{\gamma(\rho)+1} \mathcal{S} = \Delta_1\times\cdots\times\Delta_t$ is distinguished. Let $\epsilon$ be the involution on $\{1,\ldots,t\}$ supplied by Proposition \ref{first-dir}. 
Since $t$ is odd, there is a fixed point $1\leq r\leq t$ of $\epsilon$.
Hence, by Proposition \ref{first-dir}, $\Delta_r$ is distinguished. However, $\Delta_r\cong \chi^{\gamma(\rho)+1} \Delta(\rho,a,b)$ for some $a,b$. This contradicts Proposition \ref{known}.
\end{proof}

For the case of a proper ladder representation for which $\Sigma(\pi)$ is constructed from an even number of segments, the situation is more complicated. The standard module $\Sigma(\pi)$ will, in fact, be both distinguished and $\eta$-distinguished. Yet, we must show that not all of these functionals can factor through the quotient $\pi$. 
\par For that cause, we will apply the theory of Gelfand-Kazhdan derivatives for representations of $G_n$. Let us recall the mirabolic subgroup $P_n< G_n$ consisting of matrices whose bottom row is $(0\cdots 0\;1)$. For an admissible representation $\sigma$ of $G_{k-1}$, denote by $\Psi^+(\sigma)$ the representation of $P_k$ obtained by composing the natural homomorphism $P_k\to G_{k-1}$ on $\nu^{1/2}\sigma$. Also, there is a canonical functor $\Phi^+$ taking representations of $P_k$ to representations of $P_{k+1}$, whose definition we will refrain from writing here. Recall, that given $\pi\in \Pi_{G_n}$, there are well-defined finite-length (possibly zero) representations $\pi^{(k)}$ of $G_{n-k}$ for $k=1,\ldots,n$\footnote{We treat $G_0 = P_1$ as the trivial group. We will formally refer to the one-dimensional irreducible representation of the trivial group $1$, $\Psi^+(1)$ and $\Sigma(1)$ as \textit{empty} representations in our notation. The product operation on them will have a trivial meaning ($1\times \pi = \pi$).}, called the derivatives of $\pi$. There is a filtration of $\pi$ as a $P_n$-representation, such that its subquotients are isomorphic to $(\Phi^+)^{k-1}\circ\Psi^+(\pi^{(k)})$, $k=1,\ldots,n$.
\par The theory of derivatives is useful for our analysis of distinction because of the following known equality (\cite[Lemma 2.4]{AKT}) for a finite-length admissible representation $\sigma$ of $G_{n-k}$ ($1\leq k\leq n-1$):
$$\dim Hom_{P_n\cap H_n}((\Phi^+)^{k-1}\circ\Psi^+(\sigma),\mathbb{C}) = \dim Hom_{H_{n-k}}(\nu^{1/2}\sigma,\mathbb{C}).$$
\begin{lem}\label{deriv-lem}
Let $\pi\in \Pi_{G_n}$ be a non-generic distinguished representation. Then $\nu^{1/2}\sigma$ must be distinguished for at least one irreducible subquotient $\sigma$ of one of the derivatives of $\pi$.
\end{lem}
\begin{proof}
When $\pi$ is non-generic, $\pi^{(n)}$ is zero (one may take it as a definition). The non-zero $H_n$-invariant functional on $\pi$ induces a non-zero $P_n\cap H_n$-invariant functional on at least one of the $P_n$-subquotients $(\Phi^+)^{k-1}\circ\Psi^+(\pi^{(k)})$ of $\pi$, for $1\leq k\leq n-1$. The rest follows from the above equality.

\end{proof}

The so-called full derivative of a ladder representation was computed in \cite[Theorem 14]{LM}. In other words, the semisimplification of all of the derivatives of a ladder $\pi\in \Pi_{G_n}$ is known, and in fact consists of ladder representations of smaller groups. Let us repeat it here. When $\Sigma(\pi)$ is of ladder type and is given by $\Delta(\rho,a_1,b_1)\times\cdots\times \Delta(\rho,a_t,b_t)$, the set of all irreducible representations appearing as subquotients of the derivatives of $\pi$ is
$$\mathcal{D}(\pi)=\left\{\sigma\in \bigcup_{k<n}\Pi_{G_{k}}\::\: \Sigma(\sigma)\cong\Delta(\rho,a'_1,b_1)\times\cdots\times \Delta(\rho,a'_t,b_t),\,\forall i\; a_i\leq a'_i<a_{i-1} \right\},$$
where $a_0$ is taken as $\infty$ and $\Delta(\rho, a,b)=1$ if $a>b$ .

\begin{thm}\label{noboth-thm}
Let $\pi \in \Pi_{G_n}$ be a ladder representation with a decomposition $\pi\cong \pi_1\times\cdots \times\pi_k$ into proper ladder representations, as in Lemma \ref{prop-decomp}.

If $k$ is odd, $\pi$ cannot be both distinguished and $\eta$-distinguished.
\end{thm}
\begin{proof}
Suppose that $\pi=\pi_{[\rho]}\in \Pi_{G_n}$ and that $\Sigma(\pi)=\Delta_1\times\cdots\times\Delta_t$. Recalling Theorem \ref{ladder-thm}, we should prove that $\pi$ is not $\eta^{\gamma(\rho)+t}$-distinguished. If it was, then pulling back the non-zero functional would make $\Sigma(\pi)$ $\eta^{\gamma(\rho)+t}$-distinguished as well. For an odd $t$ this cannot happen by Lemma \ref{noboth-lem}. Thus, we will assume $t$ is even.

We write $k_0= (k+1)/2$. Let us denote by $i_0$ and $t_0$ the indices for which 
$$\Sigma(\pi_{k_0}) = \Delta_{i_0}\times\cdots\times \Delta_{i_0+t_0-1}\,.$$
Assume $\pi' = \chi^{\gamma(\rho)} \pi$ is distinguished. Then $(\pi')^\vee \cong (\pi')^\tau$ holds. When writing $\pi'_i = \chi^{\gamma(\rho)} \pi_i$, we clearly have $\pi'\cong \pi'_1\times \cdots\times \pi'_k$ as the decomposition of Lemma \ref{prop-decomp}. The argument in the proof of Theorem \ref{ladder-thm-2} shows that $(\pi'_{k_0})^\vee \cong (\pi'_{k_0})^\tau$ holds in such case. The same proof also shows that $t_0$ must be even as well.

In particular, $t_0>1$, which means that $\Delta_{i_0+1}$ precedes $\Delta_{i_0}$. Hence, $\Sigma(\pi')$ is not irreducible and $\pi'$ is non-generic. By Lemma \ref{deriv-lem}, there is a representation $\sigma\in \mathcal{D}(\pi')$, for which $\nu^{1/2} \sigma$ is distinguished. Hence, its standard module $\Sigma(\nu^{1/2} \sigma)$ is distinguished as well.

It can be easily seen from the description of the full derivative of a ladder representation that $\Sigma(\sigma) = \Sigma(\sigma_1)\times\cdots\times\Sigma(\sigma_k)$, for some $\sigma_i\in \mathcal{D}(\pi'_i)$. Up to some possibly empty $\sigma_i$, this is the decomposition of Lemma \ref{prop-decomp}.

From the same description, we see that $\Sigma(\sigma_{k_0})$ is constructed of either $t_0$ or $t_0-1$ segments. 

Let us first handle the case that $\Sigma(\nu^{1/2} \sigma_{k_0})=\nu^{1/2}\Sigma( \sigma_{k_0}) = \Delta'_{i_0}\times\cdots\times \Delta'_{i_0+t_0-1}$ (where all $\Delta'_i$ are non-empty segments). Arguing the same as for $\pi'$, the distinction of $\nu^{1/2} \sigma$ forces $(\nu^{1/2} \sigma_{k_0})^\vee\cong (\nu^{1/2} \sigma_{k_0})^\tau$. Now, by the proof of Theorem \ref{ladder-thm}, this means that $\Re(\Delta'_{2i_0 + t_0 -1-i})=-\Re(\Delta'_i)$ for all $i_0\leq i\leq i_0+t_0-1$. In particular, $\sum_{i=i_0}^{i_0+t_0-1} \Re(\Delta'_i)=0$. From the distinction of $\pi'$ we can again argue the same to deduce $\left(\sum_{i=i_0}^{i_0+t_0-1} \Re(\Delta_i)=\;\right) \sum_{i=i_0}^{i_0+t_0-1} \Re(\chi^{\gamma(\rho)}\Delta_i)=0$. Yet, by the description of $\mathcal{D}(\pi')$, $\Re(\Delta'_i) \geq 1/2 + \Re(\chi^{\gamma(\rho)}\Delta_i)$, for all $i_0\leq i\leq t_0+i_0-1$. Hence, a contradiction.

Otherwise, $\Sigma(\nu^{1/2} \sigma_{k_0})$ has a rang\'{e}e module realization induced from $t_0-1$ segments. Suppose that $\Sigma=\Sigma(\nu^{1/2} \sigma)$ has a rang\'{e}e module realization induced from $s$ segments. Again, referring to the same situation as in the proof of Theorem \ref{ladder-thm-2}, we see that $s-(t_0-1)$ is even. Hence, $s$ is odd. Keeping track of the tensoring operations we see that $\Sigma$ is of pure type $[\nu^{1/2}\chi^{\gamma(\rho)} \rho]$. Since $\Sigma$ is distinguished, by Lemma \ref{noboth-lem} we must have $\gamma(\nu^{1/2}\chi^{\gamma(\rho)} \rho)=0$. The remark after Proposition \ref{known} shows this is impossible, because obviously $\gamma(\chi^{\gamma(\rho)} \rho)=0$ holds.

\end{proof}


\bibliographystyle{abbrv}
\bibliography{propo2}{}

\end{document}